\documentclass[12pt,reqno]{amsart} 
\usepackage{amssymb}
\usepackage{amsthm}
\usepackage{amsmath}
\usepackage{verbatim}
\usepackage{enumerate}
\usepackage[usenames,dvipsnames] {color}
\numberwithin{equation}{section}
\newcommand\de{\delta}
\newcommand\eps{\varepsilon}
\newcommand\lam{\lambda}
\renewcommand\phi{\varphi}
\def\d{\mathbb{D}}
\def\c{\mathbb{C}}
\def\t{\mathbb{T}}
\def\r{\mathbb{R}}

\newcommand\im{\mathrm{Im~}}

\newcommand\ov{\overline}
\newcommand\half{\tfrac 12}
\def\hol{{\rm Hol}}
\newcommand\inv{^{-1}}

\newcommand\bbm{\begin{bmatrix}}
\newcommand\ebm{\end{bmatrix}}
\newcommand\bpm{\begin{pmatrix}}
\newcommand\epm{\end{pmatrix}}

\def\h{\mathcal{H}}

\def\k{\mathcal{K}}

\def\b{\mathcal{B}}
\def\m{\mathcal{M}}
\def\n{\mathcal{N}}

\def\ph{\phi}
\newcommand{\dd}{\mathrm{d}}
\newcommand\ga{\gamma}
\def\f{\mathcal{F}}

\newcommand\calm{\mathcal{M}}

\newcommand\df{\stackrel{\rm def}{=}}
\newcommand\ess{\mathcal{S}}
\newcommand\hinfty{\mathrm{H}^\infty}

\def\be{\begin{equation}}
\def\ee{\end{equation}}

\def\hol{{\rm Hol}}
\def\hinf{{\rm H}^\infty}

\def\set#1#2{\{ #1 \, : \, #2\}}

\def\norm#1{\| #1 \|}

\def\ip#1#2{\langle #1,#2\rangle}

\def\s0{s_0}
\def\p0{p_0}
\DeclareMathOperator{\ran}{ran}

\DeclareMathOperator{\re}{\rm Re}

\theoremstyle{definition}
\newtheorem{defin}[equation]{Definition}
\newtheorem{lem}[equation]{Lemma}
\newtheorem{prop}[equation]{Proposition}

\newtheorem{thm}[equation]{Theorem}

\newtheorem{example}[equation]{Example}
\newtheorem{remark}[equation]{Remark}

\begin{document}

\title[Function theory in the bfd-norm on an elliptical region]{ Function theory in the bfd-norm on an elliptical region}

\author{Jim Agler}
\address{Department of Mathematics, University of California at San Diego, CA \textup{92103}, USA}
\email{jagler@ucsd.edu}

\author{Zinaida A. Lykova}
\address{School of Mathematics,  Statistics and Physics, Newcastle University, Newcastle upon Tyne
	NE\textup{1} \textup{7}RU, U.K.}
\email{Zinaida.Lykova@ncl.ac.uk}

\author{N. J. Young}
\address{School of Mathematics, Statistics and Physics, Newcastle University, Newcastle upon Tyne NE1 7RU, U.K.}
\email{Nicholas.Young@ncl.ac.uk}

\date{2nd August, 2024}




\thanks{Partially supported by National Science Foundation Grants
	DMS 1361720 and 1665260, the Engineering and Physical Sciences Research Council grant EP/N03242X/1 and by the London Mathematical
Society grant 42321.} 

\begin{abstract}
Let $E$ be the open region in the complex plane bounded by an ellipse.  
The B. and F. Delyon norm $\|\cdot\|_{\mathrm{bfd}}$ on the space $\mathrm{Hol}(E)$ of holomorphic functions on $E$ is defined by
\[
\|f\|_{\mathrm{bfd}} \stackrel{\rm def}{=}  \sup_{T\in \mathcal{F}_{\mathrm {bfd}}(E)}\norm{f(T)},
\]
where $\mathcal{F}_{\mathrm {bfd}}(E)$ is the class of operators $T$ such that the closure of the numerical range of $T$ is contained in $E$. The name of the norm recognizes a celebrated theorem of the brothers Delyon, which implies that $\|\cdot\|_{\mathrm{bfd}}$ is equivalent to the supremum norm $\|\cdot\|_\infty$ on $\mathrm{Hol}(E)$.

The purpose of this paper is to develop the theory of holomorphic functions of bfd-norm less than or equal to one on $E$. To do so we shall employ a remarkable connection between the bfd norm on $\mathrm{Hol}(E)$ and the supremum norm $\|\cdot\|_\infty$ on the space $\mathrm{H}^\infty(G)$ of bounded holomorphic functions on the symmetrized bidisc, the  domain $G$ in $\mathbb{C}^2$ defined by
\begin{align*}
G & \stackrel{\rm def}{=} \{(z+w,zw): |z|<1, |w|<1\}.
\end{align*}
It transpires that there exists a holomorphic embedding $\tau:E \to G$ having the property that, for any bounded holomorphic function $f$ on $E$,
\[
\|f\|_{\mathrm{bfd}} = \inf\{\|F\|_\infty: F \in {\mathrm H}^\infty(G), F\circ\tau=f\},
\]
and moreover, the infimum is attained at some $F \in \mathrm{H}^\infty(G)$.  This result allows us to derive,  for holomorphic functions of bfd-norm at most one on $E$,
analogs of the well-known model and realization formulae for Schur-class functions. 
We also give a second derivation of these models and realizations, which  exploits the Zhukovskii mapping from an annulus onto $E$.

\end{abstract}

\maketitle

\section{Introduction}
In this paper we shall generalize some results from long-established function theory of the unit disc to holomorphic functions of bfd-norm at most one on an elliptical domain.
Recall that the {\em Schur class}, $\ess$, is the set of holomorphic functions $\ph$ on the unit disc $\d$ such that the supremum norm $\|\ph\|_\infty \leq 1$.

The notions of models and realizations of functions are useful for the understanding of the Schur class.
A {\em model} of a function $\ph:\d\to \c$ is a pair $(\m,u)$ where $\m$ is a Hilbert space and $u$ is a map from $\d$ to $\m$ such that, for all $\lam,\mu\in\d$,
\be\label{eqmodelD}
1-\ov{\ph(\mu)}\ph(\lam) = (1-\bar\mu \lam)\ip{u(\lam)}{u(\mu)}_\m,
\ee
where $\ip{\cdot}{\cdot}_\m$ denotes the inner product in $\m$.
A closely related notion is a {\em realization} of a function $\ph$ on $\d$, that is, a formula of the form
\be\label{realize}
\ph(\lam) =\alpha + \ip{\lam(1-D\lam)\inv \gamma}{\beta}_\m\qquad  \text{for all} \ \lam\in\d,
\ee
where $\bbm \alpha & 1\otimes \beta\\ \gamma\otimes 1 & D\ebm$ is the matrix of a  unitary operator on $\c\oplus \m$.

The connection between models, realizations and the Schur class are revealed in the following theorem.
\begin{thm}\label{1} Let $\ph$ be a function on $\d$.  The following conditions are equivalent.
\begin{enumerate}[(i)]
\item $\ph \in \ess$;
\item $\ph$ has a model;
\item $\ph$ has a realization.
\end{enumerate}
\end{thm}
Proofs of the various implications in this theorem can be found, for instance, in \cite{amy20}, where (i)$\iff$(ii) is \cite[Theorem 2.9]{amy20} and (its converse) \cite[Proposition 2.32]{amy20}, and (i)$\iff$(iii) is \cite[Theorem 2.36]{amy20}.
Models and realizations of functions have proved to be a powerful tool for both operator-theorists (e.g. Nagy and Foias \cite{NF}) and control engineers (largely as a tool for computation \cite{GL}).
In this paper we shall derive versions of model and realization formulae which apply  to functions in the ``bfd-Schur class" $\mathcal{S}_{\mathrm {bfd}}$, which we now define.

By an {\em elliptical domain} we shall mean the domain in the complex plane bounded by an ellipse. As a standard elliptical domain we take the set 
\be\label{defGde}
G_\de  \df \{x+iy:x,y\in\r, \frac{x^2}{(1+\de)^2} + \frac{y^2}{(1-\de)^2} < 1\},
\ee
for some $\de$ such that $0\leq\de<1$.
Note that any elliptical domain can be identified via an affine self-map of the plane with an elliptical domain of the form $G_\de$ for some $\de\in[0,1)$.

For a complex Hilbert space $\h$ we denote by $\b(\h)$ the space of bounded operators on $\h$. 
If  $T\in\b(\h)$, then $W(T)$, \emph{the numerical range of $T$}, is defined by the formula
\[
W(T) = \set{\ip{Tx}{x}_\h}{x\in\h, \norm{x}=1 }.
\]
For an open set $\Omega$ in $\c^n$, $\hol(\Omega)$ will denote the set of holomorphic functions defined on $\Omega$ and $\hinf(\Omega)$ will denote the Banach algebra of bounded holomorphic functions defined on $\Omega$, with pointwise operations and the supremum norm $\norm{\phi}_\infty= \sup_{z\in\Omega}|\phi(z)|$.

The {\em B. and F. Delyon family}, $\f_{\mathrm {bfd}}(C)$, corresponding to an open bounded convex set $C$ in $\c$ is the class of operators $T$ such that the closure of the numerical range of $T$, 
   $\overline{W(T)}$, is contained in $C$.  By \cite[Theorem 1.2-1]{GuRao97}, the spectrum $\sigma(T)$  of an operator $T$ is contained in $\overline{W(T)}$, and 
   so, by the Riesz-Dunford functional calculus, $\ph(T)$ is defined for all $\phi\in \hol(C)$ and $T\in \f_{\mathrm {bfd}}(C)$. Therefore, we may consider the calcular norm\footnote{A {\em calcular norm} on a function space is a norm that is defined with the aid of the functional calculus.  For more information on such norms the reader may consult \cite[Chapter 9]{amy20}.} 
\be \label{bfdnorm}
\norm{\phi}_{\f_{\mathrm {bfd}}(C)}=\sup_{T\in \f_{\mathrm {bfd}}(C)}\norm{\phi(T)},
\ee
defined\footnote{which implies that $
   \norm{\cdot}_{\f_{\mathrm {bfd}}(C)}$ is equivalent to the supremum norm $\|\cdot\|_\infty$ on $\hol(C)$.} for $\phi \in \hol(C)$, and 
the associated Banach algebra
\[
\hinf_{\mathrm {bfd}}(C)=\set{\phi \in \hol(C)}{
\norm{\phi}_{\f_{\mathrm {bfd}}(C)} <\infty}.
\]
The bfd-norm is named in recognition of a celebrated theorem \cite{delyon} of the brothers B. and F. Delyon, 
which states that, if $p$ is a polynomial, $\h$ is a Hilbert space and $T\in \b(\h)$ then 
\[
\|p(T)\| \leq \kappa(W(T))\|p\|_{W(T)},
\]
where  $\|\cdot\|_{W(T)}$ denotes the supremum norm on $W(T)$, and, for any bounded convex set $C$ in $\c$, $\kappa(C)$ is defined by
\[
\kappa(C) = 3+\left(\frac{2\pi(\mathrm{diam}(C)^2}{\mathrm{area}(C)}\right)^3.
\]
Let us write
\[
K(\f_{\mathrm {bfd}}(C)) = \sup_{\phi \in \hol(C): \|\phi\|_\infty \le 1} \norm{\phi}_{\f_{\mathrm {bfd}}(C)},
\]
and the Crouzeix universal constant
\[
K_{\mathrm {bfd}} =\sup \{ K(\f_{\mathrm {bfd}}(C)): \ C \ \text{is a  bounded convex set in } \ \c\}.
\]
In \cite{Cr2004}, Crouzeix proved $K_{\mathrm {bfd}} \le 12$ and conjectured that $ K_{\mathrm {bfd}} = 2$.
Subsequently Crouzeix and Palencia \cite{CrPal2017} proved that $ K_{\mathrm {bfd}} \leq 1+\sqrt{2}$.

In this paper the convex set $C$ will always be $G_\de$, and so we abbreviate the notation to $\|\cdot\|_{\mathrm {bfd}}$ in place of 
$\|\cdot\|_{\f_{\mathrm {bfd}}(G_\de)}$. Thus
\be\label{bfdnorm2}
\norm{\phi}_{\mathrm {bfd}}=\sup_{T\in \f_{\mathrm {bfd}}(G_\de)}\norm{\phi(T)},
\ee
defined for $\phi \in \hol(G_\delta)$.
In addition we introduce the {\em bfd-Schur class}, $\mathcal{S}_{\mathrm {bfd}}$, of functions on $G_\de$, which is the set of functions
$f \in \hol(G_\de)$ such that
$\|f\|_{\mathrm {bfd} }\leq 1.$\footnote{In the notations $\|\cdot\|_{\mathrm {bfd}}$ and $\mathcal{S}_{\mathrm {bfd}}$ we suppress dependence on the parameter $\de$.}

The following notion is  a  modification of the concept of model that is appropriate to functions in the bfd-Schur class.
\begin{defin}\label{Gdemodel-int} Let $\delta \in (0,1)$.
A {\em bfd-model~} of a function $f:G_\de \to \c$ is a triple $(\m, U,u)$, where $\m$ is a Hilbert space, $U$ is a unitary operator on $\m$ and $u:G_\de \to \m$ is a  map such that, for all $\lam, \mu \in G_\de$,
\be\label{2ndmodelform-intro}
1-\ov{f(\mu)}f(\lam) = \ip{(1-\mu_U^*\lam_U) u(\lam)}{u(\mu)}_\m,
\ee
where, for all $\lam\in G_\de$, $\lambda_U$ is the operator on $\m$ defined by
\be\label{deflamsubT-intro}
\lambda_U=(\delta U^*-\tfrac12 \lambda)(1-\tfrac12 \lambda U^*)\inv.
\ee
\end{defin}
Note that the bfd-model formula for a function $f:G_\de \to \c$ is the same as the model formula \eqref{eqmodelD} for a function $\ph:\d \to \c$ except that  $\lam,\mu$ are replaced by the operators $\lambda_U, \mu_U$ respectively, for some unitary operator $U$. 
\begin{remark}\label{aut-hol-model}
If $f$ is  an arbitrary function on $G_\de$, not assumed to be holomorphic, for which there exist  a bfd-model $(\m, U,u)$ as in Definition \ref{Gdemodel-int}, then both $u$ and $f$ are necessarily holomorphic on $G_\de$, see Proposition \ref{modprop10} and Theorem \ref{GtoGde}.
\end{remark}

For the bfd-Schur class  an  appropriate notion of realization is the following.
\begin{defin}\label{bfdrealize}
Let $f$ be a function on $G_\de$.
A {\em bfd-realization} of $f$ is a formula of the form
\be\label{thisisit-int}
f(\lam)=\alpha+\ip{\lam_U (1-D \lam_U)\inv \ga}{\beta}_\calm\quad\ \text{for all}\ \lam\in G_\de,
\ee
where $\alpha$ is a scalar, $\beta,\gamma$ are vectors in a Hilbert space $\m$ and
 $D, U$ are operators on $\calm$ such that $U$ is unitary and the operator
\be\label{collig-int}
\bbm \alpha & 1\otimes\beta \\ \ga \otimes 1& D \ebm \mbox{ is unitary on } \c \oplus \calm.
\ee
\end{defin}
Our main result is the following analog of Theorem \ref{1}. It is contained in Theorem \ref{GtoGde} and Proposition \ref{modprop10} of Section \ref{G-G-de}.
\begin{thm}\label{model-Gde} Let $\delta \in (0,1)$ and let $f$ be a
 function on $G_\de$.
The following conditions are equivalent.
\begin{enumerate}[(i)] 
\item $f$  belongs to $\mathcal{S}_{\mathrm {bfd}}$;
\item $f$ has a bfd-model;
\item $f$ has a bfd-realization.
\end{enumerate}
\end{thm}
Theorem \ref{model-Gde} is about functions of {\em one} complex variable.
One interesting aspect of our proof of this theorem is that
it exploits results about functions of {\em two} complex variables. 
Indeed, it makes use of the function theory of the symmetrized bidisc\footnote{See \cite{aly2019} for background information on  this domain.} $G$, the domain in $\c^2$ defined by
\[
G\df\{ (z+w,zw):z,w \in\d\}.
\]
It is known that an alternative presentation of $G$ is
\[
G=\{(s,p)\in\c^2: |s-\bar s p|<1-|p|^2\}.
\]
Using this characterization of $G$ one can show that $G_\de$ can be identified with a slice of the symmetrized bidisc $G$ by the holomorphic embedding $\gamma: G_\de \to G$ defined by $\gamma(\lam)=(\lam,\de)$ for all $\lam\in G_\de$. \footnote{So that $\gamma(G_\de)= G\cap \{(s,p):p=\de\}.$ }  For any $\Phi\in \hinfty(G)$ we may use $\gamma$ to define a function $\ph\in\hinfty(G_\de)$ by $\ph=\Phi\circ\gamma$.
A key result that we use to prove Theorem  \ref{model-Gde} is the following,  which is \cite[Theorem 12.13]{aly23}.
\begin{thm} \label{keyresult}
Let $\de\in (0,1)$.  For any $\phi \in \hol(G_\de)$ with a finite bfd-norm, there exists a function $\Phi\in H^\infty(G)$ such that $ \phi=\Phi\circ\gamma$, and
\be \label{min-phi}
\norm{\phi}_{\mathrm {bfd}} = \inf \{\|\Phi\|_{\infty}: \; \Phi\in H^\infty(G),\; \phi=\Phi\circ\gamma \}.
\ee
Moreover the infimum is attained.
\end{thm}

Armed with Theorem \ref{keyresult} and
the ``fundamental theorem for the symmetrized bidisc" \cite[Theorem 7.18]{amy20},
one can readily prove Theorem  \ref{model-Gde}. The fundamental theorem for the symmetrized bidisc states that any function $\Phi \in H^\infty(G)  $  with 
$\|\Phi\|_{\infty} \leq 1$  has a ``$G$-model", that is, a formula analogous to equations \eqref{eqmodelD} and \eqref{2ndmodelform-intro} but for holomorphic functions with supremum norm at most one on $G$ (see Definition \ref{defGmodel} below).  

The success of the foregoing method constitutes a piquant instance of both the discovery and the proof of new and significant facts about holomorphic functions of a single variable by deduction from known theorems about functions of two variables.

In Section \ref{G-R-de} we explore another approach to the function theory of the elliptical region $G_\de$, an approach that avoids reference to holomorphic functions of two variables.  Instead it makes use of a formally simple map that expresses the elliptical domain $G_\de$ as the image of the annulus 
\[
R_\de \df \{z\in\c: \de< |z| < 1\}
\]
under a $2$-to-$1$ covering map, to wit the map
\[
\pi(z)= z + \frac{\de}{z} \quad \text{for} \ z\in R_\de.
\] 
This map is familiar to engineers  \cite{mark,sedov} as the ``Zhukovskii mapping" or the ``Joukowski mapping". 
In this connection we shall use the following notions.
The {\em Douglas-Paulsen family } $\f_{\mathrm{dp}}(\de)$ corresponding to the annulus $R_\de$ is the class of operators $X$ such that $\|X\|\leq 1, \|X\inv\|\leq 1/\de$ and $ \sigma(X)\subseteq R_\de$. The corresponding calcular norm is
\be \label{dpnorm}
\norm{\phi}_{\mathrm {dp}} = \sup_{X\in  \f_{\mathrm{dp}}(\de) }\norm{\phi(X)},
\ee
defined for $\phi \in \hol(R_\delta)$, and the associated Banach algebra is
\[
\hinf_{\mathrm {dp}}(R_\de)=\set{\phi \in \hol(R_\delta)}{\norm{\phi}_{\mathrm {dp}}<\infty}.
\]
Our second proof of Theorem \ref{model-Gde} is accomplished with the aid of the following relationship between $\|\cdot\|_{\mathrm {dp}}$ and $\|\cdot\|_{\mathrm{bfd}}$.
\begin{thm} \label{dpbfd} Let $\de\in (0,1)$.
For all $\phi\in \hol(G_\de)$,
\[
\|\phi\|_{\mathrm {bfd}} = \|\phi\circ\pi\|_{\mathrm {dp}}.
\]
\end{thm} 
Just as a fundamental theorem for functions of norm at most one in $\hinf(G)$, combined with Theorem \ref{keyresult}, implies Theorem \ref{model-Gde}, the fundamental theorem for functions of dp-norm at most one in $\hinf_{\mathrm {dp}}(R_\de)$, combined with Theorem \ref{dpbfd}, also implies Theorem \ref{model-Gde}.

In Section \ref{Alt-models} we derive the following alternative formulation of Theorem \ref{model-Gde} as an integral formula by invoking the spectral theorem.

\begin{thm}\label{model-Gde-2} Let $\delta \in (0,1)$.
Let $f$ be a function on $G_\de$.  The following statements are equivalent.
\begin{enumerate}[(i)]
\item $f \in \mathcal{S}_{\mathrm {bfd}}$;
\item $f$ is holomorphic on $G_\de$ and
there is a triple $(\m,E,x)$, where $\m$ is a Hilbert space,
$E$ is an $\mathcal{B} (\m)$-valued spectral measure  on $\t$ and 
 $x:G_\de \to\m$ is a holomorphic map, such that, for all $\lam,\mu\in G_\de$,
  \be \label{3rd-intmodel}
  1-\overline{f(\mu)}f(\lam) = \int_\t (1-\overline{\Psi_\omega(\mu)}\Psi_\omega(\lam)) \ip{dE(\omega)x(\lam)}{x(\mu)},
\ee
where, for $\omega \in\t$, the rational function $\Psi_\omega:G_\de \to\d$, is defined  by
\[
\Psi_\omega(\lam) =  \frac{2\omega \de- \lam}{2-\omega \lam}\quad \text{for} \; \lam\in G_\de.
\]
\end{enumerate}
\end{thm}
There are two interesting ways to think of the functions $\Psi_\omega$ for  $\omega\in\t$.
Firstly, if $H_\omega$ denotes the closed supporting halfplane to the closure of the elliptical set $G_\de^-$ at the boundary point $\bar{\omega} +\de \omega$, then $\Psi_\omega$ is the linear fractional mapping that maps $H_\omega$ to the closed unit disc $\d^-$ and satisfies the two interpolation conditions $\Psi_\omega(0)=\omega\de$, $\Psi_\omega(\bar{\omega} +\de \omega)=-\omega$, see Proposition \ref{mappropPsi}.

Secondly, $\Psi_\omega=\Phi_\omega\circ\gamma$ where the functions $\Phi_\omega:G \to \c$, for $ \omega\in\t$, are defined by 
\[
 \Phi_\omega(s,p)= \frac{2\omega p-s}{2-\omega s}\quad  \text{for all} \ (s,p)\in G.
 \] 
The functions $(\Phi_\omega)_{\omega\in\t}$, which map $G$ into $\d$,  play a pervasive role in the complex geometry of $G$.
An important property of $(\Phi_\omega)_{\omega\in\t}$ is that, up to change of variable by automorphisms, it constitutes the smallest universal family  for the solution of the Carath\'eodory extremal problem on $G$.  See, for example \cite{aly2019}.
We showed in \cite{ay2008} that there is also a functional-analytic characterization of the $\Phi_\omega$: they parametrize the extreme rays of a certain cone in a space of holomorphic functions in four variables.

\section{Models of holomorphic functions on $G_\delta$} \label{G-G-de}

Let $E$ be a domain  bounded by an ellipse of positive eccentricity in the complex plane.
After a rotation and a translation $E$ can be written in the form
\be\label{defE}
E = \{x+iy:x,y\in\r, \frac{x^2}{a^2} + \frac{y^2}{b^2} < 1\},
\ee
for some $a,b \in \r$ such that $a> b>0$.
Note that the domain $E$ of equation \eqref{defE}
can be identified via the affine map $f(z)= \frac{2}{a+b} z$ with the domain
\be\label{defGde-2}
G_\de  \df \{x+iy:x,y\in\r, \frac{x^2}{(1+\de)^2} + \frac{y^2}{(1-\de)^2} < 1\},
\ee
where $0< \de= \frac{a-b}{a+b}<1$.

In \cite[Section 12]{aly23} we give an interpretation of the B. and F. Delyon norm on $\hol(G_\delta)$ in terms of the solution to a certain extremal problem in two complex variables. Let $G$ denote the symmetrized bidisc, the open set in $\c^2$ defined by either of the equivalent formulas
\[
G=\sigma(\d^2)
\]
where $\sigma:\c^2 \to \c^2$, the \emph{symmetrization map}, is defined by
\[
\sigma(z)=(z_1+z_2,z_1z_2),\qquad \text{for} \  z=(z_1,z_2)\in \c^2,
\]
or
\be\label{Gsp}
G=\set{(s,p)\in \c^2}{|s-\bar s p|<1-|p|^2}.
\ee
This domain was first studied in \cite{ay99}. Convenient summaries of its basic properties can be found in \cite{AY04}, \cite{aly2019} and \cite[Chapter 7]{jp}.
The following simple lemma demonstrates that $G_\delta$ may be viewed as a slice of $G$.
\begin{lem}\label{sym.lem.10}
Let $\delta\in (0,1)$. For any $\lam \in \c$,
\[
(\lam,\delta) \in G \iff \lam\in G_\delta.
\]
\end{lem}
\begin{proof}
Consider any $\lam\in\c$ and let $x=\re \lam $ and $y=\im \lam$.
In view of the relation \eqref{Gsp},
\begin{align*}
(\lam,\de) \in G & \iff |x+iy -(x-iy)\de| <1 -\de^2 \\
	& \iff |x(1-\de) +iy(1+\de)| < 1-\de^2 \\
	& \iff \left|\frac{x}{1+\de} +i\frac{y}{1-\de}\right| < 1 \\
		&\iff \frac{x^2}{(1+\de)^2} + \frac{y^2}{(1-\de)^2} < 1 \\
	& \iff \lam \in G_\de.	
\end{align*}
\end{proof}
Lemma \ref{sym.lem.10} implies that if $\Phi \in \hol(G)$, then we may define $\phi \in \hol(G_\delta)$ by the formula
\be\label{sym.10}
\phi(\mu)=\Phi(\mu,\delta) \quad \text{ for all } \mu \in G_\delta.
\ee
One can ask the question: Given $\phi \in H_{\mathrm {bfd}}^\infty(G_\de)$, what is the minimum $m_\phi$ of $\|\Phi\|_{H^\infty(G)}$ over all $\Phi\in H^\infty(G)$ such that $\phi$ is the restriction of $\Phi$ to $G_\de$?

The next two propositions \cite[Proposition 12.4 and Proposition 12.5]{aly23} combine to show that the minimum $m_\phi$  is $\|\phi\|_{\mathrm {bfd}}$, so that $m_\phi$ is finite if and only if $\phi\in \hinf_{\mathrm {bfd}}(G_\delta)$.

\begin{prop}\label{sym.prop.10} Let $\de \in (0,1)$. 
If $\Phi \in \hinf(G)$ and $\phi$ is defined as in equation \eqref{sym.10}, then $\phi \in \hinf_{\mathrm {bfd}}(G_\delta)$ and $\norm{\phi}_{\mathrm {bfd}} \le \norm{\Phi}_{H^\infty(G)}$.
\end{prop}

\begin{prop}\label{sym.prop.20}
If $\phi \in \hinf_{\mathrm {bfd}}(G_\delta)$, then there exists $\Phi \in \hinf(G)$ such that $\norm{\Phi}_\infty=\norm{\phi}_{\mathrm {bfd}}$ and equation \eqref{sym.10} holds.
\end{prop}

Realizations of functions on the symmetrized bidisc were studied in \cite{AY17}.
The authors introduced the following notion.
\begin{defin}\label{defGmodel}
A $G$-\emph{model} for a function $\Phi$ on $G$ is a triple $(\m,T,u)$ where $\calm$ is a Hilbert space, $T$ is a contraction acting on $\calm$ and $u:G \to \m$ is a holomorphic function such that, for all $s,t\in G$,
\be \label{modelform}
 1-\overline{\Phi(t)}\Phi(s)= \ip{ (1-t_T^* s_T) u(s)}{u(t)}_\calm.
\ee
\end{defin}
Here, for any point $s=(s^1,s^2)\in G$ and any contractive linear operator $T$ on a Hilbert space $\calm$, the operator $s_T$ is defined by
\be\label{defsU}
s_T=(2s^2T^*-s^1)(2-s^1T^*)\inv \quad \mbox{ on } \calm.
\ee
Note that $|s^1|<2$ for $s\in G$, and therefore the inverse in equation \eqref{defsU} exists.
\begin{remark}\label{aut-hol-model-2}
In the definition of a $G$-model the requirement that the map $u$ be holomorphic can be omitted:  it follows as in \cite[Proposition 2.21]{amy20} that $u$ is automatically holomorphic, and as in \cite[Proposition 2.32]{amy20} that any function on $G$ that has a $G$-model  is necessarily holomorphic.
\end{remark}

The following is the main result of the paper \cite[Theorem 2.2]{AY17}.
\begin{thm}\label{modelGthm}
Let $\Phi$ be a function on $G$.  The following three statements are equivalent.
\begin{enumerate}
\item   $\Phi\in  \hol(G)$ and $\|\Phi\|_{\infty} \leq 1$;
\item $\Phi$ has a $G$-model;
\item $\Phi$ has a $G$-model $(\calm, U, u)$ in which $U$ is a unitary operator on $\calm$.
\end{enumerate}
\end{thm}

To derive the model formula,  for functions 
$f$  on $G_\de$ with $\|f\|_{\mathrm {bfd} }\leq 1$,
we will need the following technical lemma. Recall that, for $\delta \in (0,1)$,
for any $\lam\in G_\de$ and any contractive operator $T$ we defined $\lam_T$ to be the operator given by
\be\label{deflamsubT-2}
\lambda_T=(\delta T^*-\tfrac12 \lambda)(1-\tfrac12 \lambda T^*)\inv.
\ee

\begin{lem}\label{lambda-T} Let $\delta \in (0,1)$. For any $\lam\in G_\de$ and any contractive operator $T$ on a Hilbert space $\m$, the operator $\lam_T$ on $\m$ is  well  defined and $\|\lam_T \| < 1$. Moreover, for a contractive operator $T$ on a Hilbert space $\m$,
$$g: G_\de \to \b(\m): \lam \mapsto \lam_T$$
 is a holomorphic operator-valued function on $G_\de$. 
\end{lem}
\begin{proof}
Let us first check that the definition \eqref{deflamsubT-2} is valid, that is, that $1-\tfrac12 \lambda T^*$ is invertible whenever $T$ is a contraction and $\lam\in G_\de$.  Clearly we may assume that $\lam\neq 0$.  We want to show that $2/\lam $ is in the resolvent set of $T^*$, and since $T^*$ is a contraction, it is enough to show that $\left|\frac{2}{\lam}\right| > 1$.  Now $G_\de$ is the region bounded by the ellipse centered at the origin with semi-axes of length $1+\de$ and $1-\de$, so that $|\lam|<1+\de<2$ for all $\lam\in G_\de$, and so $\left|\frac{2}{\lam}\right| > 1$ as required. Since $1-\tfrac12 \lambda T^*$ is invertible whenever $T$ is a contraction and $\lam\in G_\de$, it follows that, for every a contractive operator $T$ on a Hilbert space $\m$,  the operator-valued function  
\[g: G_\de \to \b(\m): \lam \mapsto \lam_T\]
 is  holomorphic on $G_\de$.

Now $\lam_T=f_\lam(T^*)$ where 
\[
f_\lam(z)=\frac {2\de z-\lambda}{2-\lambda z}
\]
 for $z$ in a neighborhood of the closure of $\d$.   The linear fractional map $f_\lam$ maps $\d$ onto the open disc with center and radius
\[
 2\frac{\overline{\lambda}\delta-\lambda}{4-|\lambda|^2} \quad \mbox{ and } \quad \frac{|\lambda^2-4\delta|}{4-|\lambda|^2} \quad\text{respectively.}
\]
Therefore, by von Neumann's inequality,
\be \label{Gineq}
\|\lam_T\| = \|f_\lam(T^*)\| \leq \sup_\d |f_\lam| = \frac{2|\lambda-\bar{\lambda}\delta|+|\lambda^2-4\delta|}{4-|\lambda|^2}.
\ee
We claim that the right-most term in formula \eqref{Gineq} is less than one for $\lam \in G_\de$.
To see this, recall that, by \cite[Theorem 2.1]{AY04}, for any point $(s,p) \in G$ 
we have
\be\label{Gineq2}
\frac{2|s-\bar sp|+|s^2-4p|}{4-|s|^2} < 1.
\ee
By Lemma \ref{sym.lem.10}, a complex number $\lam$ belongs to $G_\de$ if and only if $(\lam,\de) \in G$.  Hence, for any $\lam \in G_\de$, inequality \eqref{Gineq2} applies with $(s,p) = (\lam,\de)$, which is to say that 
\[
\frac{2|\lambda-\bar{\lambda}\delta|+|\lambda^2-4\delta|}{4-|\lambda|^2} < 1,
\]
as claimed.
Hence $\|\lam_T\| < 1$.

\end{proof}

Let us prove that models for $\mathcal{S}_{\mathrm {bfd}}$ are automatically holomorphic.
\begin{prop}\label{modprop10} {\it (Automatic holomorphy of models)}
Let $\delta \in (0,1)$. Let $f:G_\de \to \c$  be a function and suppose that $f$ has a bfd-model $(\m, U,u)$ as in Definition \ref{Gdemodel-int}. Then 
  $u:G_\de \to \m$ is holomorphic on $G_\de$.
\end{prop}
\begin{proof} By assumption, $f:G_\de \to \c$ has a bfd-model $(\m, U,u)$, 
 that is, there is a triple $(\m,U,u)$ where $\m$ is a Hilbert space, $U$ is a unitary operator on $\m$ and $u:G_\de \to\m$ is a  map, such that, for all $\lam,\mu\in G_\de$,
\be\label{2ndmodelform-4-intro}
1-\ov{f(\mu)}f(\lam) = \ip{(1-\mu_U^*\lam_U) u(\lam)}{u(\mu)}_\m.
\ee
Rearrange the above formula to obtain
\[
1+\ip{\lam_U u(\lambda)}{\mu_U u(\mu)}=\ov{f(\mu)}f(\lam)+\ip{u(\lambda)}{u(\mu)},
\]
which is to say that the two families of vectors
\[
\bpm 1 \\ \lambda_U u(\lam)\epm_{\lam\in G_\de} \quad \mbox{ and }\quad \bpm f(\lam) \\ u(\lam) \epm_{\lam\in G_\de}
\]
in $\c \oplus \calm$ have the same Gramians.  Hence there exists an isometry
$V$ between the closed linear spans of the above two families,
\be\label{prop-V1}
V: \ov{\text{span}} \left\{ \bpm 1 \\ \lam_U u(\lam) \epm : \lam \in G_\de \right\}   
 \to 
\ov{\text{span}} \left\{ \bpm f(\lam) \\ u(\lam) \epm :  \lam \in G_\de \right\},
\ee
 such that
\be\label{prop-V2}
V\bpm 1 \\ \lam_U u(\lam) \epm = \bpm f(\lam) \\ u(\lam) \epm \qquad  \text{for every}\ \lam \in G_\de.
\ee
  If necessary enlarge the Hilbert space $\calm$ (and simultaneously the unitary operator $U$ on $\calm$) so that the isometry $V$ extends to a unitary operator
\[
 V^\sharp \sim \bbm \alpha & 1\otimes\beta \\ \ga\otimes 1& D \ebm \mbox{ on } \c \oplus\calm
\]
for some vectors $\beta,\ga \in\calm$ and some operator $D$ on $\m$.  By equation \eqref{prop-V2}, for any $\lam\in G_\de$,
\begin{align}\label{hol-2eq}
\alpha+\ip{\lam_U u(\lam)}{\beta}&=f(\lam), \notag\\
\ga+D\lam_U u(\lam)&=u(\lam).
\end{align}
As $V^\sharp$ is a contraction, so also is $D$. Consequently, we may solve the last equation in the system \eqref{hol-2eq} for $u(\lam)$ to obtain the relation
\be
u(\lam)  = (1 -D \,\lambda_U)^{-1} \gamma  \quad\mbox{ for all }{\lambda \in G_\de}.
\ee
By Lemma \ref{lambda-T}, $g: G_\de \to \b(\m): \lam \mapsto \lam_U$
 is a holomorphic operator-valued function on $G_\de$ and $\|\lam_U\| <1$, and so $1-D\lam_U$ is invertible for any $\lam \in G_\de$.
This implies that $u$ is  holomorphic on $G_\de$.
\end{proof}

\begin{prop}\label{modprop20}  Let $\delta \in (0,1)$. Let $f:G_\de \to \c$  be a function and suppose that $f$ has a bfd-model $(\m, U,u)$. Then 
  $f:G_\de \to \c$ is holomorphic on $G_\de$.
\end{prop}
\begin{proof} If $f(\mu) =0$ for all $\mu \in G_\de $, then clearly, $f$ is holomorphic on $\d$. Otherwise, choose $\mu\in G_\de$ such that $f(\mu) \not=0$ and solve the model formula 
\be\label{2ndmodelform-5-intro}
1-\ov{f(\mu)}f(\lam) = \ip{(1-\mu_U^*\lam_U) u(\lam)}{u(\mu)}_\m
\ee
for $f(\lambda)$ to obtain
\be\label{hol-f}
f(\lambda) = \overline{f(\mu)}^{\ -1} \Big(1-\ip{(1-\mu_U^*\lam_U) u(\lam)}{u(\mu)}_\m \Big) \quad \mbox{ for all }\lambda \in G_\de.
\ee
By Proposition \ref{modprop10},  $u:G_\de \to \m$ is  holomorphic on $G_\de$, and
by Lemma \ref{lambda-T}, $g: G_\de \to \b(\m): \lam \mapsto \lam_U$
 is a holomorphic operator-valued function on $G_\de$.
Therefore,  equation \eqref{hol-f} implies that $f$ is holomorphic on $G_\de$.
\end{proof}

\begin{thm} \label{GtoGde}
Let $\delta \in (0,1)$.
Let $f$ be a function on $G_\de$.  The following statements are equivalent.
\begin{enumerate}[(i)]
\item $f \in \mathcal{S}_{\mathrm {bfd}} $;
\item there exists $\Phi\in\hinf (G)$ such that $\|\Phi\|_\infty \leq 1$ and $f(\lam)=\Phi(\lam,\de)$ for all $\lam \in G_\de$;
\item there is a triple $(\m,U,x)$, where $\m$ is a Hilbert space, $U$ is a unitary operator on $\m$ and $x:G_\de \to\m$ is a holomorphic map, such that, for all $\lam,\mu\in G_\de$,
\be\label{2ndmodelform-intro-2}
1-\ov{f(\mu)}f(\lam) = \ip{(1-\mu_U^*\lam_U) x(\lam)}{x(\mu)}_\m,
\ee
where, for $\lam \in G_\de$,  $\lambda_U$ is the operator defined by
Definition \ref{deflamsubT-2};
\item
there exist a scalar $\alpha$, a Hilbert space $\calm$, vectors $\beta,\gamma \in \calm$ and operators 
 $D, U$ on $\calm$ such that $U$ is unitary, the operator
\be\label{collig-intro}
\bbm \alpha & 1\otimes\beta \\ \ga \otimes 1& D \ebm \mbox{ is unitary on } \c \oplus \calm
\ee
and, for all $\lam \in G_\de$,
\be\label{thisisit-intro}
f(\lam)=\alpha+\ip{\lam_U (1-D \lam_U)\inv \ga}{\beta}_\calm.
\ee
\end{enumerate}
\end{thm}
\begin{proof}
We shall show that (i) $\implies$ (ii) $\implies$ (iii) $\implies$ (iv)$\implies$ (i).\\

(i) $\implies$ (ii) is Proposition \ref{sym.prop.20} (see \cite[Proposition 12.5]{aly23}). \\

Assume (ii) holds.  Choose a function $\Phi\in\hinf (G)$ such that $\|\Phi\|_\infty \leq 1$ and $f(\lam)=\Phi(\lam,\de)$ for all $\lam \in G_\de$.  By Theorem \ref{modelGthm},  $\;\Phi$ has a $G$-model $(\calm, U, u)$. Thus
there is a triple $(\m,U,x)$, where $\m$ is a Hilbert space, $U$ is a unitary operator on $\m$ and $u:G \to\m$ is a holomorphic map, such that, for all $s=(s^1,s^2), t=(t^1,t^2) \in G$,
\be \label{modelform-2}
 1-\overline{\Phi(t^1,t^2)}\Phi(s^1,s^2)= \ip{ (1-t_T^* s_T) u(s)}{u(t)}_\calm.
\ee
Consider points $\lam,\mu\in G_\de$. By Lemma \ref{sym.lem.10}, the points $s=(\lam, \de), t=(\mu, \de)$ belong to $G$. We have chosen $\Phi\in\hinf (G)$  such that $f(\lam)=\Phi(\lam,\de)$ for all $\lam \in G_\de$. Therefore the formula \eqref{modelform-2} specializes to 
\be \label{modelform-3}
1-\ov{f(\mu)}f(\lam) = \ip{(1-\mu_U^*\lam_U) u(\lam, \de)}{u(\mu, \de)}_\m
\ee
for all $\lam,\mu\in G_\de$, since 
\be\label{defsU-2}
s_U=(\lam, \de)_U=(2\de U^*-\lam)(2-\lam U^*)\inv = \lam_U \quad \mbox{ on } \calm.
\ee
Let $x:G_\de \to\m$ be defined by $x(\lam) =  u(\lam, \de)$.  Then equation \eqref{modelform-3} shows that $(\m,U,x)$ is a $G_\de$-model of $f$.
The proof that (ii) $\implies$ (iii)  is complete. \\

Assume (iii) holds. That is,
 there is a triple $(\m,U,x)$ where $\m$ is a Hilbert space, $U$ is a unitary operator on $\m$ and $x:G_\de \to\m$ is a holomorphic map, such that, for all $\lam,\mu\in G_\de$,
\be\label{2ndmodelform-3}
1-\ov{f(\mu)}f(\lam) = \ip{(1-\mu_U^*\lam_U) x(\lam)}{x(\mu)}_\m.
\ee
Rearrange the above formula to obtain
\[
1+\ip{\lam_U x(\lambda)}{\mu_U x(\mu)}=\ov{f(\mu)}f(\lam)+\ip{x(\lambda)}{x(\mu)},
\]
which is to say that the two families of vectors
\[
\bpm 1 \\ \lambda_U x(\lam)\epm_{\lam\in G_\de} \quad \mbox{ and }\quad \bpm f(\lam) \\ x(\lam) \epm_{\lam\in G_\de}
\]
in $\c \oplus \calm$ have the same Gramians.  Hence there exists an isometry $L$ between the closed linear spans of the above two families of vectors
\be\label{propL1}
L: \ov{\text{span}} \left\{ \bpm 1 \\ \lam_U x(\lam) \epm : \lam \in G_\de \right\}   
 \to 
\ov{\text{span}} \left\{ \bpm f(\lam) \\ x(\lam) \epm :  \lam \in G_\de \right\}
\ee
 such that
\be\label{propL2}
L\bpm 1 \\ \lam_U x(\lam) \epm = \bpm f(\lam) \\ x(\lam) \epm
\ee
for every ${\lam \in G_\de}$.  If necessary enlarge the Hilbert space $\calm$ (and simultaneously the unitary operator $U$ on $\calm$) so that the isometry $L$ extends to a unitary operator
\[
 L^\sharp \sim \bbm \alpha & 1\otimes\beta \\ \ga\otimes 1& D \ebm \mbox{ on } \c \oplus\calm
\]
for some vectors $\beta,\ga \in\calm$ and some operator $D$ on $\m$.  By equation \eqref{propL2}, for any $\lam\in G_\de$,
\begin{align}\label{2eq}
\alpha+\ip{\lam_U x(\lam)}{\beta}&=f(\lam), \notag\\
\ga+D\lam_U x(\lam)&=x(\lam).
\end{align}
By Lemma \ref{lambda-T}, the operator $\lam_U$ is well defined and $\|\lam_U\| < 1$, and so $1-D\lam_U$ is invertible for any $\lam \in G_\de$, and we may eliminate $x(\lam)$ from equations \eqref{2eq} to obtain the realization formula \eqref{thisisit-int} for $f$.  Thus (iii)$\implies$(iv).

It remains to prove that (iv)$\implies$(i).  Suppose (iv): there exist a scalar $\alpha$, a Hilbert space $\calm$, vectors $\beta,\gamma \in \calm$ and operators 
 $D, U$ on $\calm$ such that $U$ is unitary, the operator
\be\label{collig-int-2}
\bbm \alpha & 1\otimes\beta \\ \ga \otimes 1& D \ebm \mbox{ is unitary on } \c \oplus \calm
\ee
and, for all $\lam \in G_\de$,
\be\label{thisisit-int-2}
f(\lam)=\alpha+\ip{\lam_U (1-D \lam_U)\inv \ga}{\beta}_\calm.
\ee
By Lemma \ref{lambda-T}, $g: G_\de \to \m: \lam \mapsto \lam_U$
 is a holomorphic operator-valued function on $G_\de$ and $\|\lam_U\| <1$.
 Hence $f$ is a holomorphic function on $G_\de$.

By \cite[Theorem 1.1]{AY17}, if there exist  a scalar $\alpha$, a Hilbert space $\calm$, vectors $\beta,\gamma \in \calm$ and operators 
 $D, U$ on $\calm$ such that $U$ is unitary, the operator
\be\label{collig-int-3}
\bbm \alpha & 1\otimes\beta \\ \ga \otimes 1& D \ebm \mbox{ is unitary on } \c \oplus \calm,
\ee
then the function $\Phi$ on $G$ defined by
\begin{align*}
\Phi(s) &= \alpha+ (1\otimes\beta)s_U(1-Ds_U)\inv(\gamma\otimes 1) \\
	&=  \alpha + \ip{s_U(1-Ds_U)\inv\gamma}{\beta}\qquad \text{for all}\ s \in G
\end{align*}
satisfies $\Phi \in \hinf(G)$ and $\|\Phi\|_\infty \leq 1$. 
Here $s_U$ is defined by equation \eqref{defsU}. 
 Now by Proposition \ref{sym.prop.10}, the function $g\in\hol(G_\de)$ defined by $g(\lam) =\Phi(\lam,\de)$ for all $\lam\in G_\de$ satisfies $g\in \hinf_{\mathrm{bfd}}(G_\de)$ and $\|g\|_{\mathrm{bfd}} \leq 1$.  Note that 
\[
g(\lam) = \Phi(\lam,\de) = \alpha + \ip{(\lam,\de)_U(1-D(\lam,\de)_U)\inv\gamma}{\beta}
\; \text{for all}\ \lam \in G_\de.
\]
By equation \eqref{defsU-2}, $(\lam,\de)_U = \lam_U$, and so
\[
g(\lam)= \alpha + \ip{\lam_U(1-D\lam_U)\inv\gamma}{\beta},
\]
which shows that the function $f$, which is defined by the same equation as $g$, to wit equation \eqref{thisisit-int-2}, satisfies $\|f\|_{\mathrm{bfd}} \leq 1$.  Hence (iv) $\implies$ (i).
\end{proof}

\begin{remark}\label{War-Lagr} {\rm The Waring-Lagrange theorem:} ~Any symmetric polynomial in $d$ commuting variables can be expressed as a polynomial in the $d$ elementary symmetric functions in the variables.

In an interesting historical survey \cite[pp. 364-365]{Funk}, H. G. Funkhauser attributes this fundamental fact about symmetric functions to Edward Waring, Lucasian Professor at Cambridge \cite{War} in 1770 and independently to Joseph Lagrange \cite{Lag} in 1798.  Lagrange's account is later, but is more explicit, so we shall call the statement the {\em Waring-Lagrange theorem}.  It is likely that Euler had the result around the same time.

Although we have stated the Waring-Lagrange theorem only for symmetric polynomials, we actually invoke it for symmetric holomorphic functions on a symmetric open set.  It is not difficult to deduce the result for symmetric holomorphic functions from the statement for polynomials, with the aid of the fact that a symmetric holomorphic function on a symmetric open set $U$ can be approximated locally uniformly on $U$ by symmetric polynomials.

\end{remark}

\begin{example} \label{rat-funct-in-Sbfd} {\it Examples of rational functions in 
$\mathcal{S}_{\mathrm {bfd}} $.}\\
Here we will describe a way to construct rational functions in 
$\mathcal{S}_{\mathrm {bfd}} $.
Let $P$ be a symmetric polynomial in $2$ variables such that $P(z,w) \neq 0$ for all $ z, w \in \d$. Suppose $P$ has degree $n$ in $z$ and in $w$. We define the polynomial $ \widetilde{P}$ by the formula
\[ \widetilde{P}(z,w) = z^n w^n \overline{P \left( \frac{1}{\bar{z}}, \frac{1}{\bar{w }}
\right)}.
\]
Since $P(z,w) \neq 0$ for all $ z, w \in \d$,
the function
 $\widetilde{P}/P$ is a symmetric rational inner function on $\d^2$, so that 
 $\sup_{\d^2} |\widetilde{P}/P| \leq 1$, see
\cite{rudinpoly}. By the Waring-Lagrange theorem, there exists a holomorphic rational function $f$ on $G$ such that 
\[
f\circ\rho =\widetilde{P}/P \ \text{on}\ \d^2,
\]
where $\rho:\d^2 \to G$ is defined by $\rho(z,w)=(z+w,zw)$.  Therefore $\sup_G |f| \leq 1$.  Define $\phi: G_\de \to \c$ by $\phi(\lam)=f(\lam, \de)$ for all $\lam\in G_\de$. By Proposition \ref{sym.prop.10}, 
$$\|\phi\|_{\mathrm{bfd}} \leq \sup_G |f| \leq 1.$$

\end{example}

\begin{example} \label{exbfd} 
{\it Examples of functions in $\phi\in\hol(G_\de)$ for which the bfd norm and supremum norm coincide.}

Clearly $\|\phi\|_{\mathrm{bfd}} \geq \|\phi\|_{\hinf(G_\de)}$ for any $\phi\in\hol(G_\de)$. It is obvious that, for any constant function $\phi\in\hol(G_\de)$, we have $\|\phi\|_{\mathrm{bfd}} = \|\phi\|_{\hinf(G_\de)}$.
It can happen that $\|\phi\|_{\mathrm{bfd}} = \|\phi\|_{\hinf(G_\de)}$ in other cases, for example, it is so if $\phi=\Psi_\omega$ for any $\omega\in\t$.  
Recall that $\Psi_\omega$ is the restriction to $G_\de \times \{\de\}$ of the ``magic function" $\Phi_\omega\in \hol(G)$, defined by
\[
\Phi_\omega(s,p) = \frac{2\omega p-s}{2-\omega s} \ \text{for}\ (s,p)\in G.
\]
In Lemma \ref{Phi1} we show that
\[
\|\Psi_\omega\|_{\mathrm{bfd}} = 1 = \|\Psi_\omega\|_{\hinf(G_\de)}\ \text{for all} \ \omega\in\t.
\]
\end{example}

In Example \ref{id-bfd} below we prove that, for the function $\phi(z) =z$, $\|\phi\|_{\mathrm{bfd}} =2$, but $\|\phi\|_{\hinf(G_\de)} = 1+\de$.

\section{Models of holomorphic functions on $G_\delta$ via the Zhukovskii mapping} \label{G-R-de} 

  In this section we give another approach to the function theory of an elliptical
domain. We begin by deriving a model formula for certain holomorphic functions on the elliptical domain $G_\de$,
starting from a model given in \cite[Chapter 9]{amy20} for suitable holomorphic functions on the annulus $R_\de$.  We shall set out the relevant definitions from \cite{amy20}.  As above we fix $\de \in (0,1)$, and we retain the notations $R_\de =\{z\in\c: \de<|z|<1\}$ and $\pi(z)=z+\de/z$ for $z\in R_\de$.  Recall that $G_\de = \pi(R_\de)$. 
Recall also that the {\em Douglas-Paulsen family} of operators and the {\em dp-norm} on $\hol(R_\de)$ were defined in the Introduction (see equation \eqref{dpnorm}).

\cite[Theorem 9.46]{amy20} states the following.
\begin{thm}\label{dpmodel}
Let $\delta \in (0,1)$. Let $\phi:R_\de\to\c$ be holomorphic and satisfy $\|\phi\|_{\f_{\mathrm {dp}}(\de)} \leq 1$. There exists a $\mathrm {dp}$-model $(\n,v)$ of $\phi$ with parameter $\de$, in the sense that there are Hilbert spaces $\n^+,\n^-$
such that $\n=\n^+ \oplus \n^-$  and an ordered pair $v=(v^+,v^-)$ of holomorphic functions, where $v^+:R_\de \to \n^+$ and $v^-:R_\de\to \n^-$ satisfy, for all $z,w \in R_\de$,
\[
1-\overline{\phi(w)}\phi(z) = (1-\bar w z)\ip{v^+(z)}{v^+(w)}_{\n^+} + (\bar w z-\de^2)\ip{v^-(z)}{v^-(w)}_{\n^-}.
\]
\end{thm}

Observe that $\mathrm {dp}$-models are models of holomorphic functions on the annulus $R_\de$.  From the existence of $\mathrm {dp}$-models we shall derive the existence of models of certain holomorphic functions on the elliptical domain $G_\de$.  An appropriate notion of model for such functions is the following.

\begin{defin}\label{Gdemodel} Let $\delta \in (0,1)$.
A {\em $G_\de$-model} of a function $f:G_\de \to \c$ is a triple $(\m, U,u)$, where $\m$ is a Hilbert space, $U$ is a unitary operator on $\m$ and $u:G_\de \to \m$ is a holomorphic map such that, for all $\lam, \mu \in G_\de$,
\be\label{6.2}
1-\ov{f(\mu)}f(\lam) = \ip{\left(2-2\de^2-(\bar\mu-\de \lam)U- (\lam-\de \bar\mu)U^*\right)u(\lam)}{u(\mu)}.
\ee
\end{defin}

Now observe that if $\phi \in \hol(G_\delta)$ then we may define $\pi^\sharp(\phi) \in \hol(R_\delta)$ by the formula
\[
\pi^\sharp(\phi)(\lambda)=\phi(\pi(\lambda)) \qquad\text{ for all } \lambda\in R_\delta.
\] 
The following two results are Lemma 11.24 and Theorem 11.25 in \cite{aly23}.

\begin{lem}\label{dp.lem.10} Let $\de \in (0,1)$ and let
 $\psi \in \hol(R_\delta)$. Then $\psi \in \ran \pi^\sharp$ if and only if $\psi$ is \emph{symmetric} with respect to the involution $\lam\mapsto\de/\lam$ of $R_\de$, that is, if and only if $\psi$ satisfies
\[
\psi(\delta/\lambda)=\psi(\lambda)
\]
for all $\lambda \in R_\delta$.
\end{lem}

The following theorem from \cite{aly23} gives an intimate connection between the $\norm{\cdot}_{\mathrm {dp}}$ and $\norm{\cdot}_{\mathrm {bfd}}$ norms.

\begin{thm}\label{dp.thm.10}  Let $\de \in (0,1)$.
The mapping $\pi^\sharp$ is an isometric isomorphism from $\hinf_{\mathrm {bfd}}(G_\de)$ onto the set of symmetric functions with respect to the involution $\lam\mapsto\de/\lam$ in $\hinf_{\mathrm {dp}}(R_\de)$,  so that, for all $\phi\in \hol(G_\de)$,
\be\label{bfd=dp}
\|\phi\|_{\mathrm {bfd}} = \|\phi\circ\pi\|_{\mathrm {dp}}.
\ee
\end{thm}

Our next result guarantees the existence of $G_\de$-models.

\begin{thm}\label{Gdemodelsexist} Let $\de\in (0,1)$.
If $f \in \mathcal{S}_{\mathrm {bfd}} $
 then $f$ has a $G_\de$-model.
\end{thm}

\begin{proof}
By Theorem \ref{dp.thm.10},
\[
\|f\circ\pi\|_{\mathrm {dp}}=\|f\|_{\mathrm {bfd}} \leq 1.
\]
Hence by Theorem \ref{dpmodel}, $f\circ\pi$ has a $\mathrm {dp}$-model $(\n,v)$ with parameter $\de$, say $\n=\n^+\oplus\n^-$ and $v=(v^+,v^-)$, so that, for all $z,w\in R_\de$,
\be\label{1.1}
1-\ov{f\circ\pi(w)}f\circ\pi(z) = (1-\bar w z)\ip{v^+(z)}{v^+(w)}_{\n^+} +(\bar w z-\de^2) \ip{v^-(z)}{v^-(w)}_{\n^-}
\ee
Since $\pi(\de/z)=\pi(z)$, we may replace $z,w$ by $\de/z,\de/w$ in equation \eqref{1.1} to obtain
\be\label{1.2}
1-\ov{f\circ\pi(w)}f\circ\pi(z) = \left(1-\frac{\de^2}{\bar w z}\right)\ip{v^+(\de/z)}{v^+(\de/w)}_{\n^+} +\left(\frac{\de^2}{\bar w z}-\de^2 \right) \ip{v^-(\de/z)}{v^-(\de/w)}_{\n^-}.
\ee
Add equations \eqref{1.1} and \eqref{1.2} and divide by $2$ to obtain, for all $z,w\in R_\de$,
\begin{align}\label{avge}
1-\ov{f\circ\pi(w)}&f\circ\pi(z) = \half(1-\bar w z)\ip{v^+(z)}{v^+(w)}_{\n^+} +\half(\bar w z-\de^2) \ip{v^-(z)}{v^-(w)}_{\n^-}\\
	&\hspace*{0.5cm}  +\half \left(1-\frac{\de^2}{\bar w z}\right)\ip{v^+(\de/z)}{v^+(\de/w)}_{\n^+} +\half \left(\frac{\de^2}{\bar w z}-\de^2\right) \ip{v^-(\de/z)}{v^-(\de/w)}_{\n^-} \notag \\
		&=\hspace*{0.5cm} \half(1-\bar w z)\left[ \ip{v^+(z)}{v^+(w)}_{\n^+} + \ip{\frac{\de}{z}v^-\left(\frac{\de}{z} \right)}{\frac{\de}{w}v^-\left(\frac{\de}{w}\right)}_{\n^-}\right] \notag \\
		&\hspace*{1cm} + \half\left(1-\frac{\de^2}{\bar w z}\right)\left[ \ip{zv^-(z)}{wv^-(w)}_{\n^-}+\ip{v^+\left(\frac{\de}{z} \right)}{v^+\left(\frac{\de}{w}\right)}_{\n^+}\right]. \label{avge.2}
\end{align}
Let $\tilde v:R_\de \to \n^+\oplus\n^-$ be defined by
\[
\tilde v(z)= \frac{1}{\sqrt{2}}\bpm v^+(z) \\ \frac{\de}{z} v^-\left(\frac{\de}{z} \right) \epm \text{ for all } z\in R_\de.
\]
Then we may re-write equation \eqref{avge.2} in the form
\be \label{2.3}
1-\ov{f\circ\pi(w)} f\circ\pi(z) = (1-\bar w z)\ip{\tilde v(z)}{\tilde v(w)}_\n +\left(1-\frac{\de^2}{\bar w z}\right)\ip{\tilde v\left(\frac{\de}{z} \right)}{\tilde v\left(\frac{\de}{w}\right)}_\n \text{ for all } z,w\in R_\de.
\ee
In equation \eqref{2.3} replace $z$ by $\de/z$ (but leave $w$ unchanged) and use the fact that $\pi(\de/z)=\pi(z)$ to obtain the equation 
\begin{align}\label{3.0}
(1-\bar w z)\ip{\tilde v(z)}{\tilde v(w)}_\n &+\left(1-\frac{\de^2}{\bar w z}\right)\ip{\tilde v\left(\frac{\de}{z} \right)}{\tilde v\left(\frac{\de}{w}\right)}_\n  \\
&= 1-\ov{f\circ\pi(w)} f\circ\pi(z) \notag\\
&= 1-\ov{f\circ\pi(w)} f\circ\pi(\de/z)  \notag \\
&=(1-\bar w \de/z)\ip{\tilde v\left(\frac{\de}{z} \right)}{\tilde v(w)}_\n +(1-\frac{\de}{\bar w} z)\ip{\tilde v(z)}{\tilde v\left(\frac{\de}{w}\right)}_\n. \label{3.1}
\end{align}
Thus the difference between the expression \eqref{3.1} and the left-hand side of equation \eqref{3.0} is zero.  This fact can be written
\begin{align}
\ip{\tilde v(z)}{\tilde v(w)}_\n &+ \ip{\tilde v\left(\frac{\de}{z} \right)}{\tilde v\left(\frac{\de}{w}\right)}_\n - \ip{\tilde v\left(\frac{\de}{z} \right)}{\tilde v(w)}_\n -\ip{\tilde v(z)}{\tilde v\left(\frac{\de}{w}\right)}_\n =
\bar w z\ip{\tilde v(z)}{\tilde v(w)}_\n \notag \\ &+\frac{\de^2}{\bar w z}\ip{\tilde v\left(\frac{\de}{z} \right)}{\tilde v\left(\frac{\de}{w}\right)}_\n - \bar w \frac{\de}{z}\ip{\tilde v\left(\frac{\de}{z} \right)}{\tilde v(w)}_\n - \frac{\de}{\bar w}z\ip{\tilde v(z)}{\tilde v\left(\frac{\de}{w}\right)}_\n. \label{3.2}
\end{align}
Notice that both sides of equation \eqref{3.2} factor, to give the relation
\begin{align} \label{3.4}
\ip{\tilde v(z)- \tilde v\left(\frac{\de}{z} \right)}{\tilde v(w) - \tilde v\left(\frac{\de}{w}\right)}_\n = \ip{z\tilde v(z)- \frac{\de}{z}\tilde v\left(\frac{\de}{z} \right)}{w\tilde v(w) - \frac{\de}{w}\tilde v\left(\frac{\de}{w}\right)}_\n
\end{align}
for all $z,w \in R_\de$.
That is to say, the Gramian of the family of vectors $(\tilde v(z)- \tilde v\left(\frac{\de}{z} 
\right))_{z\in R_\de}$ in $\n$ is equal to the Gramian of the family of vectors $(z\tilde v(z)- \frac{\de}{z}\tilde v\left(\frac{\de}{z} \right))_{z\in R_\de}$ in $\n$.  By the Lurking isometry lemma \cite[Lemma 2.18]{amy20}, there exists a linear isometry 
\[
L:\ov{\text{span}} \left \{ \tilde v(z)- \tilde v\left(\frac{\de}{z} \right): z\in R_\de \right \} \to 
\ov{\text{span}}\left\{z\tilde v(z)- \frac{\de}{z}\tilde v\left(\frac{\de}{z}\right):z\in R_\de \right\}
\]
 such that, for all $z\in R_\de$,
\[
L\left(\tilde v(z)- \tilde v\left(\frac{\de}{z}\right)\right) = z\tilde v(z)- \frac{\de}{z}\tilde v\left(\frac{\de}{z}\right).
\]
Since $L$ is an isometry on $\n$, we may embed $\n$ isometrically in a Hilbert space $\m$, possibly of larger cardinality, in such a way that $L$ admits an extension $U$ that is unitary on $\m$ (see \cite[Remark 2.31]{amy20}).  Then
\[
U\left(\tilde v(z)- \tilde v\left(\frac{\de}{z}\right)\right) = z\tilde v(z)- \frac{\de}{z}\tilde v\left(\frac{\de}{z}\right)
\]
for all $z\in R_\de$, and therefore
\[
(U-z)\tilde v(z) = \left(U- \frac{\de}{z}\right)\tilde v(\frac{\de}{z}),
\]
whence
\[
\left(U- \frac{\de}{z}\right)\inv\tilde v(z) = (U-z)\inv \tilde v \left(\frac{\de}{z}\right) 
\]
for all $z\in R_\de$.  Define a map $	\tilde u:R_\de \to \m$ by
\be \label{4.4}
\tilde u(z) = (U-z)\inv\tilde v\left(\frac{\de}{z}\right) = \left(U- \frac{\de}{z}\right)\inv\tilde v(z)
\ee
for all $z \in R_\de$.  Then
\[
\tilde v(z)=\left(U-\frac{\de}{z}\right)\tilde u(z), \quad  \tilde v\left(\frac{\de}{z}\right)=(U-z)\tilde u(z).
\]
By equation \eqref{2.3}, for all $z, w \in R_\de$,
\begin{align}
1-\ov{ f\circ\pi (w) } f\circ\pi(z) &= \left(1-\bar{w} z \right) \ip{ \left(U-\frac{\de}{z} \right) \tilde u(z) }
{ \left(U-\frac{\de}{w}\right)\tilde u(w)}_\m \notag \\
~ & \;\;\; + \left(1-\frac{\de^2}{\bar{w} z}\right) \ip{(U-z)\tilde u(z)}{(U-w)\tilde u(w)}_\m \notag\\
~	& = \left(1-\bar{w} z \right) \ip{\left(U^*-\frac{ \de }{ \bar{w} } \right) \left(U-\frac{\de}{z}\right)\tilde u(z)}{\tilde u(w)}_\m \notag\\ 
~	& \;\;\; + \left(1-\frac{\de^2}{\bar w z}\right) \ip{(U^*-\bar w)(U-z)\tilde u(z)}{\tilde u(w)}_\m \notag\\
		&=\ip{Z\tilde u(z)}{\tilde u(w)}_\m, \label{5.1}
\end{align}
where 
\begin{align}
Z&= (1-\bar w z)\left(U^*-\frac{\de}{\bar w}\right) \left(U-\frac{\de}{z}\right)+\left(1-\frac{\de^2}{\bar w z}\right) \left(U^*-\bar w \right)(U-z) \notag\\ &= (1-\bar w z)\left(1-\frac{\de}{\bar w} U -\frac{\de}{z} U^* +\frac{\de^2}{\bar w z}\right) +\left(1-\frac{\de^2}{\bar w z}\right)(1-\bar wU-zU^*+ \bar w z) \notag \\
	&= 2-2\de^2 -\left[\frac {\de}{\bar w}+\bar w -\de\left(z+\frac{\de}{z}\right)\right]U - \left[\frac{\de}{z}+z-\de \left(\bar w+ \frac{\de}{\bar w}\right) \right]U^*. \label{6.0}
\end{align}
Now $\tilde u$ is holomorphic on $R_\de$ and is symmetric with respect to the involution $z \mapsto \de/z$, meaning that $\tilde u(z)=\tilde u(\frac{\de}{z})$ for all $z\in R_\de$.  Hence $\tilde u$ factors through $G_\de$, in the sense that there is a holomorphic map $u:G_\de \to \m$ such that $\tilde u=u\circ\pi$.  We may therefore re-write equations \eqref{5.1} and \eqref{6.0} in terms of the symmetrized variables
\[
\lam = \pi(z)=z+\frac{\de}{z},\quad  \mu=\pi(w)=w+\frac{\de}{w},
\]
which range over $G_\de$ as $z,w$ range over $R_\de$.  We deduce that, for all $\lam,\mu \in G_\de$,
\be \label{phmodel}
1-\ov{f(\mu)}f(\lam) = \ip{\left(2-2\de^2-(\bar \mu-\de\lam)U -(\lam-\de\bar\mu)U^*\right)u(\lam)}{u(\mu)}_\m.
\ee
That is, according to Definition \ref{Gdemodel}, $(\m, U,u)$ is a $G_\de$-model of $f$.
\end{proof}

\section{Alternative models of functions with bfd-norm at most $1$.}\label{Alt-models}

In this section we show that the models described in Theorem \ref{Gdemodelsexist} can also be expressed in slightly different forms.  For this purpose we shall need some properties of the functions $\Phi_\omega: G \to \c$, where $\omega \in\t$, given by
\[
\Phi_\omega(s,p) = \frac{2\omega p-s}{2-\omega s}\quad \text{for} \ (s,p)\in G. 
\]
By \cite[Theorem 2.1 and Corollary 2.2]{AY04}, $\Phi_\omega(G)\subset \d$ for any $\omega\in \t$.  Note that, for any $\de\in (0,1)$ and $\lam\in G_\de$, we have $(\lam,\de)\in G$, and so we may define $\Psi_\omega:G_\de \to\d$ by
\be \label{defPside}
\Psi_\omega(\lam) = \Phi_\omega(\lam,\de)\; \text{for} \; \lam\in G_\de.
\ee

The functions $\Psi_\omega$ so defined are characterized by their mapping properties relative to the elliptical region $G_\de$, as described in Proposition \ref{mappropPsi} below. To prove this proposition we
shall make use of the \v{Z}ukovskii mapping 
\[
\check{Z}(z)=\frac{1}{z}+\de z
\]
 which maps the annulus 
 \[
 A_\de \df\{z\in\c: 1< |z| < 1/\de\} 
 \]
   in a 2-to-1 manner to the elliptic region $G_\de$, except for the two points $\pm2\sqrt{\de}\in G_\de$, which are the foci of the ellipse $\partial G_\de$. The preimages in $A_\de$ under $\check{Z}$ of the two foci are the singleton sets $\{ \frac{1}{\sqrt{\de}}$\} and $\{ -\frac{1}{\sqrt{\de}}$\}. Every other point of $G_\de$ has exactly two preimages in $A_\de$.
   
   For any point $\omega\in\t$, we shall denote by $H_\omega$ the closed supporting halfplane to the elliptical set $G_\de$ at the boundary point $\check{Z}(\omega)$ and by $\ell_\omega$ the tangent line to $G_\de$ at $\check{Z}(\omega)$. 
\begin{prop} \label{mappropPsi}
Let $0<\de<1$ and let $\omega\in\t$. The function $\Psi_\omega$ defined in equation \eqref{defPside} is the unique linear fractional mapping that maps $H_\omega$ onto $\d^-$, maps $0$ to $\omega\de$ and maps $\check{Z}(\omega)$ to $-\omega$.
\end{prop}
\begin{proof}
The function  $\check{Z}$ maps both bounding circles of $A_\de$ bijectively onto the ellipse $\partial G_\de$, and so determines a parametrization of  $\partial G_\de$:
\[
\check{Z}(e^{i\theta})=e^{-i\theta} + \de e^{i\theta} = (1+\de)\cos \theta - i (1-\de)\sin \theta\]
for $0\leq \theta < 2\pi$.  Note that as $\theta$ increases from $0$ to $2\pi$, $\check{Z}(e^{i\theta})$ describes the ellipse $\partial G_\de$ in a clockwise sense, that is, in a negative orientation.  Hence
\[
\frac{\dd}{\dd\theta} \check{Z}(e^{i\theta}) = i e^{i\theta}\check{Z}'(e^{i\theta})=
i e^{i\theta}(-e^{-2i\theta} +\de)
\]
is a tangent vector to the ellipse $\partial G_\de$ at the point $\check{Z}(e^{i\theta})$ pointing in the negative orientation.  It follows that an inward-pointing normal to the ellipse at the point $\check{Z}(e^{i\theta})$ is 
\[
-i\frac{\dd}{\dd\theta} \check{Z}(e^{i\theta}) = -e^{-i\theta} +\de e^{i\theta}.
\]
Note that, for any point $\omega=e^{i\theta} \in \t$, the halfplane $H_\omega$ is bounded by the tangent line $\ell_\omega$  to $G_\de$ at $\check{Z}(\omega)= \bar\omega+\de\omega$, and the tangent line itself, the boundary of $H_\omega$, has direction $-i\bar\omega+i\de\omega$.  Thus
\[
\ell_\omega = \{\check{Z}(\omega)+ti(-\bar\omega+\de\omega):t\in\r\} = \{\bar\omega+\de\omega +ti(-\bar\omega+\de\omega):t\in\r\}.
\]
Thus $z \in \c$ belongs to 
$\ell_\omega$ if and only if
\be\label{lomega}
z= \bar\omega+\de\omega +ti(-\bar\omega+\de\omega) \ \text{for some} \  t\in\r.
\ee
It can be checked that, for every $ \omega \in \t$ and for $z \in \ell_\omega$ as in equation \eqref{lomega},
\be \label{ell-T}
\Psi_\omega(z) = \frac{2 \omega \delta - z}{2 -\omega z} = - \bar{\omega} \frac{1 -it} { 1 +it} \in\t.
\ee
Thus the linear fractional mapping $\Psi_\omega$ maps $\ell_\omega$  
to the unit circle and the set $G_\de$ to $\d$. Hence it maps the closed 
halfplane $H_\omega$ onto $\d^-$.

It is easy to check that $\Psi_\omega$
 maps $0$ to $\omega\de$ and maps $\check{Z}(\omega)= \bar\omega+\omega\de$ to $-\omega$.
Thus $\Psi_\omega$ has the claimed mapping properties.
Furthermore, there is a {\em unique} linear fractional map of the plane that maps a prescribed line to a prescribed circle and satisfies interpolation conditions at two points.  Hence $\Psi_\omega$ is  the unique linear fractional mapping that maps $H_\omega$ onto $\d^-$, maps $0$ to $\omega\de$ and maps $\check{Z}(\omega)$ to $-\omega$. 
\end{proof}

Consider a unitary operator $U$ on a Hilbert space $\m$.  By the spectral theorem (for example, \cite[Theorem 1.18]{amy20}), there exists a $\b(\m)$-valued spectral measure $E$ on $\t$ such that
\be\label{specrep}
U= \int_\t \omega \ dE(\omega).
\ee
\begin{lem}\label{modelmeas}
Let $U$ be a unitary operator on a Hilbert space $\m$.
Suppose $U^*$  has spectral representation $U^*=\int_\t \omega\  dE(\omega)$,
for some $\mathcal{B} (\m)$-valued spectral measure $E$ on $\t$.
Then, for any $\lam,\mu \in G_\de$,
 \be \label{intmodel}
1-\mu_U^*\lam_U = \int_\t 1-\overline{\Psi_\omega(\mu)} \Psi_\omega(\lam)\  dE(\omega).
\ee
\end{lem}
\begin{proof}
By properties of spectral integrals, for any $\lam\in G_\de$,
\[
\lam_U = \frac{2\de U^*-\lam}{2-\lam U^*}= \int_\t \frac{2\de \omega-\lam}{2-\lam \omega}\  dE(\omega) = \int_\t \Psi_\omega(\lam)\  dE(\omega),
\]
and therefore,
\[
\mu_U^* = \int_\t \overline{\Psi_\omega(\mu)}\  dE(\omega).
\]
Hence, by the multiplicativity of spectral integrals,
for any $\lam,\mu \in G_\de$,
 \be \label{intmodel1}
\mu_U^*\lam_U = \int_\t \overline{\Psi_\omega(\mu)} \Psi_\omega(\lam)\  dE(\omega).
\ee
Equation \eqref{intmodel} follows.
\end{proof}

\begin{lem}\label{Phi1} Let $\de\in (0,1)$ and $T \in \f_{\mathrm{bfd}}(G_\de)$ act on a Hilbert space $\h$. For every  $\omega \in\t$,
\be \label{PhiT}
\|\Psi_\omega(T) \| \leq 1.
 \ee
 Moreover
 \[
\|\Psi_\omega\|_{\mathrm{bfd}} = 1 = \|\Psi_\omega\|_{\hinf(G_\de)}\ \text{for all} \ \omega\in\t.
\]
\end{lem}
\begin{proof}
Recall that  $\Psi_\omega:G_\de \to\d$ is defined by
\[
\Psi_\omega(\lam) = \Phi_\omega(\lam,\de)\; \text{for} \; \lam\in G_\de.
\]
By \cite[Proposition 12.4]{amy20}, 
\[
\|\Psi_\omega\|_{\mathrm {bfd}} \leq 
 \sup_{(s,p) \in G}|\Phi_\omega(s, p)|.
\]
By \cite[Theorem 2.1]{AY04}, for every  $\omega \in\t$,
\[\sup_{(s,p) \in G}|\Phi_\omega(s, p)| \leq 1.\]
 Hence,
 $\|\Psi_\omega\|_{\mathrm{bfd}}\leq \|\Phi_\omega\|_{\hinf(G)} \leq 1$. 
  
  Let us show that $\|\Psi_\omega\|_{\mathrm{bfd}} =1$.
 Clearly $\|\Psi_\omega\|_{\mathrm{bfd}} \geq \|\Psi_\omega\|_{\hinf(G_\de)}$. 
For any $\omega\in\t$, choose $r\in (0,1)$ and let
 $s=r\bar\omega+r\omega\de$.  It is known that, for any choice of $\beta,p \in\d$, the point $(\beta +\bar\beta p,p) \in G$. Thus, if we take $\beta=r\bar\omega, p=\de$ then we deduce that $(r\bar\omega+r\omega\de,\de)\in G$ and so $r\bar\omega+r\omega\de\in G_\de$, that is, $s\in G_\de$.  Then
\[
|\Psi_\omega(s)| = \left|\frac{2\omega \de-s}{2-\omega s}\right| =  \left|\frac{2\omega \de-r\bar\omega -r\omega\de}{2-r-r\omega^2\de}\right| \to\left|\frac{\bar\omega(\omega^2\de-1)}{1-\omega^2\de}\right|  = \left|-\bar\omega\right|= 1 \ \text{as} \ r \to 1.
\]
Therefore $\|\Psi_\omega\|_{\hinf(G_\de)} = 1$, and so
\[
\|\Psi_\omega\|_{\mathrm{bfd}} = 1 = \|\Psi_\omega\|_{\hinf(G_\de)}\ \text{for all} \ \omega\in\t.
\]

By \cite[Theorem 1.2-1]{GuRao97}, the spectrum $\sigma(T)$  of any operator $T$ is contained in $\overline{W(T)}$. Hence, for every $T \in \f_{\mathrm{bfd}}(G_\de)$,
$\sigma(T) \subseteq \overline{W(T)} \subseteq G_\delta$, and so $\Psi_\omega(T)$ is well defined and 
$\; \|\Psi_\omega(T) \| \leq 1$.
\end{proof}
The following  implication is stated and proved in \cite[Lemma 7.38]{amy20} in the case that  $\Omega=G$. 
\begin{lem} \label{hereditary-calcul}
Let $\k,\h$ be Hilbert spaces and let $\Omega$ be a domain in $\c^d$, let $k$ be a positive semi-definite $\b(\k)$-valued hereditary function on $\Omega$, let $T=(T_1,\dots,T_d)$ be a commuting $d$-tuple of operators on $\h$ such that $\sigma(T) \subseteq \Omega$ and let $h$ be a scalar-valued hereditary function on $\Omega$. If $h(T) \geq 0$ then $(hk)(T) \geq 0$. 
\end{lem}
The proof given in \cite[Lemma 7.38]{amy20} remains valid if one substitutes $\Omega$ for $G$ throughout. We shall apply the lemma to $\Omega=G_\de$ in the next theorem.

\begin{thm}\label{2nd-model} Let $\delta \in (0,1)$.
Let $f$ be a function on $G_\de$.  The following statements are equivalent.
\begin{enumerate}[(i)]
\item $f \in \mathcal{S}_{\mathrm {bfd}} $;
\item $f$ has a $G_\de$-model;
\item there is a triple $(\m,U,x)$, where $\m$ is a Hilbert space, $U$ is a unitary operator on $\m$ and $x:G_\de \to\m$ is a holomorphic map, such that, for all $\lam,\mu\in G_\de$,
\be\label{2ndmodelform-4}
1-\ov{f(\mu)}f(\lam) = \ip{(1-\mu_U^*\lam_U) x(\lam)}{x(\mu)}_\m
\ee
where, for $\lam \in G_\de$, $\lambda_U$ is the operator defined by
Definition \ref{deflamsubT-2};
\item $f$ is holomorphic on $G_\de$ and
there is a triple $(\m,E,x)$, where $\m$ is a Hilbert space,
$E$ is an $\mathcal{B} (\m)$-valued spectral measure  on $\t$ and 
 $x:G_\de \to\m$ is a holomorphic map, such that, for all $\lam,\mu\in G_\de$,
  \be \label{4th-intmodel}
  1-\overline{f(\mu)}f(\lam) = \int_\t (1-\overline{\Psi_\omega(\mu)}\Psi_\omega(\lam)) \ip{dE(\omega)x(\lam)}{x(\mu)}.
\ee
 
\end{enumerate}
\end{thm}
\begin{proof} 
We shall show that (i) $\implies$ (ii) $\implies$ (iii) $\implies$ (iv) $\implies$ (i).\\

(i) $\implies$ (ii).  Assume (i) holds.
By Theorem \ref{Gdemodelsexist}, $f$ has a $G_\de$-model $(\m,U,u)$, so that, for all $\lam,\mu \in G_\de$, equation \eqref{phmodel} holds, that is,
\begin{align}\label{newmodel}
1-\ov{f(\mu)}f(\lam) &= 
\ip{ \left(2-2\de^2-(\bar \mu-\de\lam)U -(\lam-\de\bar\mu)U^* \right) u(\lam)}{u(\mu)}_\m .
\end{align}

(ii) $\implies$ (iii).
Note that, by inspection,  for a unitary operator $U$ and  for all $\lam,\mu\in G_\de$,
\be\label{inspec}
2-2\de^2-(\bar \mu-\de\lam)U -(\lam-\de\bar\mu)U^*
= 2 (1-\tfrac12 \mu U^*)^* (1-\mu_U^* \lambda_U)(1-\tfrac12 \lambda U^*).
\ee
Therefore, by equations \eqref{newmodel} and \eqref{inspec},
\begin{align}\label{newmodel-2}
1-\ov{f(\mu)}f(\lam) &= 
 2\ip{(1-\mu_U^*\lam_U)(1-\half\lam U^*)u(\lam)}{(1-\half\mu U^*)u(\mu)}_\m .\;\; 
\end{align}
For $\lam \in G_\de$, let 
\be\label{defx}
x(\lambda) =\sqrt{2} (1-\half \lambda U^*) u(\lambda).
\ee
Then $x:G_\de \to \m$ is holomorphic, and  equation \eqref{newmodel} can be written, for all $\lam,\mu \in G_\de$, 
\begin{align*}
1-\overline{f(\mu)}f(\lam) &=\ip{(1-\mu_U^* \lambda_U)x(\lambda)}{x(\mu)}_\m,
\end{align*}
so that (iii) holds.  Thus (ii) implies (iii).

(iii) $\implies$ (iv).  Suppose that (iii) holds, so that there exist
a Hilbert space $\m$, a unitary operator $U$ on $\m$ and a holomorphic map $x:G_\de \to\m$, such that, for all $\lam,\mu\in G_\de$,
\be\label{2ndmodelform-4-1}
1-\ov{f(\mu)}f(\lam) = \ip{(1-\mu_U^*\lam_U) x(\lam)}{x(\mu)}_\m.
\ee
By Lemma \ref{lambda-T}, $\lam_U$ is a holomorphic operator-valued function of $\lam$ 
and $x(\cdot)$ is holomorphic by assumption, and so, for fixed $\mu\in G_\de$,
the right hand side of equation  \eqref{2ndmodelform-4-1}  is a holomorphic function of $\lam$.
Hence $f$ is holomorphic on $G_\de$.

By Lemma \ref{modelmeas}, for the unitary operator $U$ on $\m$, there is a $\b(\m)$-valued spectral measure $E$ on $\t$ such that, 
for any $\lam,\mu \in G_\de$,
 \be \label{intmodel-2}
1-\mu_U^*\lam_U = \int_\t 1-\overline{\Psi_\omega(\mu)} \Psi_\omega(\lam)\  dE(\omega).
\ee
On combining equations \eqref{intmodel-2} and \eqref{2ndmodelform-4-1} we deduce that, for all $\lam,\mu \in G_\de$,
\be\label{E-model}
1-\overline{f(\mu)}f(\lam) = \int_\t (1-\overline{\Psi_\omega(\mu)}\Psi_\omega(\lam)) \ip{dE(\omega)x(\lam)}{x(\mu)}.
\ee

(iv) $\implies$ (i).
Assume (iv), that is, $f$ is a holomorphic function on $G_\de$ and there exists a triple $(\m,E,x)$, where $\m$ is a Hilbert space, $E$ is a $\b(\m)$-valued spectral measure on $\t$ and $x:G_\de \to \m$ is a holomorphic map such that, for all $\lam,\mu\in G_\de$,
\be\label{E-model2}
1-\overline{f(\mu)}f(\lam) = \int_\t (1-\overline{\Psi_\omega(\mu)}\Psi_\omega(\lam)) \ip{dE(\omega)x(\lam)}{x(\mu)}.
\ee

To prove that $\|f\|_{\mathrm{bfd}} \leq 1$ we shall first demonstrate the following. \\
{\bf Step 1:}  {\em
$1- \overline{f(\mu)}f(\lam)$ can be approximated uniformly on compact subsets of $G_\de \times G_\de$ by finite sums of functions of the form
\[
 \left(1-\overline{\Psi_\omega(\mu)}\Psi_\omega(\lam)\right) k(\lam,\mu)
\]
where $\omega \in\t$ and $k$ is a positive semi-definite hereditary function on $G_\de$. } \\
  To do this let $\eps > 0$ and consider a compact subset $K$ of $G_\de$.  Let $M=\sup_{\lam\in K}\|x(\lam)\|$.
By the uniform continuity of the map $(s,\omega) \mapsto \Psi_\omega(s)$ from $K \times \t$ to $\c$ there exists $\de'>0$ such that 
\[
|\ov{\Psi_{\omega_1}(t)}\Psi_{\omega_1}(s) - \ov{\Psi_{\omega_2}(t)}\Psi_{\omega_2}(s)| < \tfrac 14 \eps/M^2
\]
for all $s,t\in K$ whenever $\omega_1,\omega_2 \in\t$ and $|\omega_1- \omega_2|< \de'$.

Choose a partition $\t= I_1\cup I_2\cup\dots\cup I_N$ of $\t$ into half-open intervals of length $<\de'$ and choose a point $\omega_j^* \in I_j$ for $j=1,\dots,N$.
By equation \eqref{E-model2}, for $\lam,\mu\in K$,
\begin{align}\label{2.1}
~ & \left|1-\ov{f(\mu)}f(\lam) -
\sum_{j=1}^N (1-\ov{\Psi_{\omega_j^*}(\mu)}\Psi_{\omega_j^*}(\lam)) \ip{E(I_j)x(\lam)}{x(\mu)} \right| = \notag  \\
  &  \left|\sum_{j=1}^N \int_{I_j} (1-\ov{\Psi_\omega(\mu)}\Psi_\omega(\lam))\ip{dE(\omega)x(\lam)}{x(\mu)} - (1-\ov{\Psi_{\omega_j^*}(\mu)}\Psi_{\omega_j^*}(\lam)) \ip{E(I_j)x(\lam)}{x(\mu)} \right| =  \notag\\
 &  \left|\sum_{j=1}^N \int_{I_j} (1-\ov{\Psi_\omega(\mu)}\Psi_\omega(\lam) - (1-\ov{\Psi_{\omega_j^*}(\mu)}\Psi_{\omega_j^*}(\lam)) \ip{dE(\omega)x(\lam)}{x(\mu)} \right| \leq  \notag \\
  & \sum_{j=1}^N \left| \int_{I_j}(\ov{\Psi_{\omega_j^*}(\mu)}\Psi_{\omega_j^*}(\lam) -\ov{\Psi_\omega(\mu)}\Psi_\omega(\lam) ) \ip{dE(\omega)x(\lam)}{x(\mu)} \right|   <   \notag
   \\
    & \sum_{j=1}^N \tfrac 14 \eps M^{-2} \times \text{the variation of}\ \ip{E(\cdot)x(\lam)}{x(\mu)}\  \text{on} \ I_j =  \notag \\
     &  \tfrac 14 \eps M^{-2} \times \text{the total variation of}\ \ip{E(\cdot)x(\lam)}{x(\mu)}\  \text{on}\  \t.
\end{align}
Since $E(\cdot)$ is projection-valued we have 
\[
|\ip{E(\tau)x(\lam)}{x(\mu)}| \leq \|x(\lam)\| \|x(\mu)\| \leq M^2
\]
for any measurable set $\tau\subseteq \t$ and any $\lam,\mu \in K$,
and therefore the total variation of the complex scalar measure $\ip{E(\cdot)x(\lam)}{x(\mu)}\  \text{on}\  \t$ is at most $4 M^2$ for any $\lam,\mu \in K$.  Hence, by the inequality \eqref{2.1},
\be\label{2.2}
\left|1-\ov{f(\mu)}f(\lam) -
\sum_{j=1}^N (1-\ov{\Psi_{\omega_j^*}(\mu)}\Psi_{\omega_j^*}(\lam)) \ip{E(I_j)x(\lam)}{x(\mu)} \right| < \eps
\ee
for all $\lam,\mu \in K$. 

Let $u_j:G_\de \to \ran E(I_j)$ be defined by $u_j(\lam)= E(I_j)x(\lam)$ for $\lam\in G_\de$ and $j=1,\dots,N$. Let also $h_j$ be the hereditary scalar function $h_j(\lam,\mu)= 1-\ov{\Psi_{\omega_j^*}(\mu)}\Psi_{\omega_j^*}(\lam)$.  Then inequality \eqref{2.2} can be written
\[
\left |1 - \ov{f(\mu)}f(\lam) - \sum_{j=1}^N h_j(\lam,\mu) \ip{u_j(\lam)}{u_j(\mu)}\right| < \eps
\]
for all $\lam,\mu\in K$.  Define the function $k_j$ on $G_\de\times G_\de$ by 
\be\label{kernel}
k_j(\lam,\mu)\df \ip{u_j(\lam)}{u_j(\mu)}
\ee 
and note that $k_j$ is a positive semi-definite hereditary function. Thus  $1-\ov{f(\mu)}f(\lam)$ can be approximated uniformly on compact subsets of $G_\de\times G_\de$ by finite sums of functions of the form 
\[\left(1-\overline{\Psi_\omega(\mu)}\Psi_\omega(\lam)\right) k(\lam,\mu),\] where $\omega\in\t$.

{\bf Step 2.} Consider any $T\in \f_{\mathrm{bfd}}(G_\de)$ and $\omega\in\t$, and so $\sigma(T)\subseteq G_\de$.
By Lemma \ref{Phi1}, for any operator $T \in \f_{\mathrm{bfd}}(G_\de)$,  $\|\Psi_\omega(T) \| \leq 1$. Thus
$\left(1-\Psi_\omega(T)^*\Psi_\omega(T)\right)$ is positive.
By Lemma \ref{hereditary-calcul}, for a domain $G_\de$, 
for a hereditary function
\[h(\lam,\mu)= 1 - \ov{\Psi_\omega(\mu)}\Psi_\omega(\lam), \;\; \lam,\mu\in G_\de,\] 
and for a positive semi-definite hereditary function $k$ defined by \eqref{kernel}
on $G_\de$, we have $(hk)(T)\geq 0$.  Therefore, for  $\omega\in\t$,  for a hereditary function 
 \[g(\lam,\mu) = \left(1-\overline{\Psi_\omega(\mu)}\Psi_\omega(\lam)\right) k(\lam,\mu),\] 
and for $T\in \f_{\mathrm{bfd}}(G_\de)$, we have $g(T) \geq 0$.

{\bf Step 3 - conclusion.}  By Step 1, $1-\ov{f(\mu)}f(\lam)$ is a limit, uniformly on compact subsets of $G_\de\times G_\de$, of functions $g$ which are finite sums of functions of the form 
\[\left(1-\overline{\Psi_\omega(\mu)}\Psi_\omega(\lam)\right) k(\lam,\mu),\] for some $\omega\in\t$ and some positive semi-definite hereditary function $k$ on $G_\de$.  By Step 2, for every 
$T\in \f_{\mathrm{bfd}}(G_\de)$, $g(T)\geq 0$ for each such $g$. By the continuity of the hereditary functional calculus, $1-f(T)^*f(T)$ is the limit in the operator norm of a sequence of positive operators.  Hence, for every $T\in \f_{\mathrm{bfd}}(G_\de)$, $1-f(T)^*f(T)\geq 0$, and so $\|f(T)\|\leq 1$.  Thus
\[
\|f\|_{\mathrm{bfd}} = \sup_{T\in  \f_{\mathrm{bfd}}(G_\de)}\|f(T)\| \leq 1.
\]
\end{proof}

Let $\delta \in (0,1)$. We introduce also the closed elliptical set $K_\delta$ defined by 
\[
K_\delta = \set{x+iy}{x,y\in\r, \frac{x^2}{(1+\delta)^2} +\frac{y^2}{(1-\delta)^2} \le 1}.
\]

\begin{example} \label{id-bfd} \rm
Let us consider the function $f\in\hol (G_\de)$ defined by $f(z)=z$.   By a slight abuse of notation we shall simply denote this function by $z$.
Let us prove that  
\[
\| z \|_{\mathrm{bfd}} = 2.
\]
\begin{proof} Recall that, by equation \eqref{bfdnorm}, 
\[
\norm{z}_{\mathrm {bfd}}=\sup_{T\in \f_{\mathrm {bfd}}(G_\de)}\norm{T}.
\]
By \cite[Proposition 11.16]{aly23},
for $\phi\in\hol(G_\de)$, the following equation holds:
\be\label{usefulbfd}
\|\phi\|_{\mathrm {bfd}} = \sup_{T\in \f_{\mathrm {bfd}}(G_\de), \; T\text{ is a matrix}} \|\phi(T)\|.
\ee
Consider any $T$ in $\b(\h)$, where $\h$ is a finite-dimensional Hilbert space, such that $\overline{W(T)} \subseteq G_\de$.  By \cite[Theorem 10.6]{aly23} there exists $X \in \b(\h\oplus\h)$ such that $\|X\| \leq 1, \|X\inv\| \leq 1/\de$ and $T$ is the restriction of $X+\de X\inv$ to $\h\oplus\{0\}$.  Hence
\[
\|T\| \leq \|X+\de X\inv\| \leq \|X\| + \|\de X\inv\| \leq 2.
\]
Thus
\[
\| z \|_{\mathrm{bfd}} = \sup_{\overline{W(T)} \subseteq G_\de} \|T\| \leq 2.
\]
On the other hand, consider the operator $X$ which corresponds to the matrix  $\bbm \sqrt{\de} & 1-\de \\ 0 & -\sqrt{\de} \ebm$ acting on $\c^2$. We have $\|X\| =1$ and $X^2=\de$, so that $\de X\inv= X$. 
Explicit calculation in Example 3 of \cite[Section 1]{GuRao97}) shows that if 
$T = (1- \varepsilon)( X+\de X\inv)$, for some small $\varepsilon >0$, then  the 
numerical range  of $T$ is $ (1- \varepsilon) K_\de$, and so $\overline{W(T)} \subseteq G_\de$.
Therefore, for small $\varepsilon >0$,
\[
 \| z \|_{\mathrm{bfd}} \geq \|T\| = (1- \varepsilon)\|X+\de X\inv\| =(1- \varepsilon)\|2X\|=2(1- \varepsilon).
 \]
Thus  $\| z \|_{\mathrm{bfd}} = 2$.
\end{proof}
\end{example}

\end{document}